\documentclass[11pt,thmsa]{article}%
\usepackage{amsmath}
\usepackage{amsfonts}
\usepackage{amssymb}
\usepackage{graphicx}%
\newtheorem{theorem}{Theorem}[section]

\newtheorem{definition}[theorem]{Definition}

\newtheorem{lemma}[theorem]{Lemma}

\newtheorem{proposition}[theorem]{Proposition}

\newenvironment{proof}[1][Proof]{\noindent\textbf{#1.} }{\ \rule{0.5em}{0.5em}}

\begin{document}

\title{Clifford-Wolf translations of Homogeneous Randers spheres\footnote{Supported by NSFC (no, 10671096, 10971104) and SRFDP of China}}
\author{  Shaoqiang Deng$^{1}$ and Ming Xu$^2$
\thanks{Corresponding author. E-mail: mgxu@math.tsinghua.edu.cn}\\
$^1$School of Mathematical Sciences and LPMC\\
Nankai University\\
Tianjin 300071, P. R. China\\
$^2$Department of Mathematical Sciences\\
Tsinghua University\\
Beijing 100084, P. R. China}
\date{}
\maketitle
\begin{abstract}
In this paper, we study Clifford-Wolf translations of homogeneous Randers metrics on spheres. It turns out that we can present a complete description
of all the Clifford-Wolf translations of all the homogeneous Randers metrics on spheres.
The most important point of this paper is that a new phenomena surfaces. Namely, we find that there are some CW-homogeneous
Randers spaces which are essentially not symmetric. This is a great difference compared to Riemannian geometry, where any CW-homogeneous
Riemannian manifold must be locally symmetric.

\textbf{Mathematics Subject Classification (2000)}: 22E46, 53C30.

\textbf{Key words}: Finsler spaces, Clifford-Wolf translations, Killing vector fields, homogeneous Randers manifolds.
\end{abstract}

\section{Introduction}
In this paper we continue our study concerning Clifford-Wolf translations of Finsler spaces in our previous article (\cite{DXP}). Our main goal here is to give a complete description of Clifford-Wolf translations of homogeneous Randers metrics on spheres. Recall that a Clifford-Wolf of a locally compact connected metric space is an isometry of the space which moves all the point in the same distance. Although generically the distance function of a Finsler space is not reversible, one can similarly define the Clifford-Wolf translation of a Finsler space (see Definition 2.2 below).

A connected Riemannian manifold $(M, Q)$ is called Clifford-Wolf homogeneous (CW-homogeneous) if for any $x,y\in M$, there is a Clifford-Wolf translation $\sigma$ such that $\sigma (x)=y$. It is called restrictively CW-homogeneous if for ant $x\in M$, there is a open neighborhood $V$ of $x$ such that for any $x'\in V$ there is an isometry $\sigma'$ of $(M,Q)$ such that $\sigma'(x)=x'$. CW-homogeneous Riemannian manifolds were thoroughly studied by V. N. Berestovskii and Yu. G. Nikonorov in \cite{BN09}. It was proved in \cite{BN09} that any restrictively CW-homogeneous Riemannian manifold must be locally symmetric. Based on this, the authors of \cite{BN09} obtained a complete classification of all connected simply connected CW-homogeneous Riemannian manifolds. The complete list consists of   compact Lie groups with bi-invariant Riemannian metrics, the odd-dimensional spheres with standard metrics and the symmetric space $SU(2n+2)/Sp(n)$ with the standard symmetric metrics.

The notion of CW-homogeneous and restrictively CW-homogeneous Riemannian manifold can be generalized to the Finsler case (see Definition 2.5 below). It is therefore a natural problem to find out whether the above conclusions still hold for Finsler spaces and to give a complete classification of all the CW-homogeneous Finsler spaces. This  problem is much more difficult compared to the Riemannian case. To begin with we first consider homogeneous Randers metrics on spheres. The main results of this paper is that there are some homogeneous Randers metrics on the spheres which are restrictively  CW-homogeneous but which are essentially not locally symmetric, in the following senses: The underlying Riemannian metrics of such Randers metrics are not locally symmetric and; Such Randers metrics are not of Berwald type. Note that a locally symmetric Finsler space must be of Berwald type, this fact was conjectured by the first author and Z. Hou in \cite{DH07} and was recently proved by  V. S. Matveev and M. Troyanov in \cite{MTP}. Meanwhile, we will give a complete list of all the Clifford-Wolf translations of  the homogeneous Randers metrics on spheres.

\section{Preliminary}
Finsler geometry is introduced by Riemann in 1854 in his celebrated  lecture on the foundations of geometry,  and revived in 1918 by Finsler
in his doctoral dissertation.
\begin{definition}
A Finsler metric on a manifold $M$ is a continuous function
$F:TM\rightarrow \mathbb{R}$, which is smooth on $TM\backslash 0$,
and satisfies the following conditions:

(1) (Positivity) $F(x,y)>0$ if $y\neq 0$.

(2) (Positive homogeneousness) $F(x,\lambda y)=\lambda F(x,y)$ for $\lambda>0$.

(3) (Convexity) The Hessian matrices of $F^2$ for $y$,
i.e. $g_{ij}=\frac{1}{2}[F^2]_{y^i y^j}$, are positively definite on
$TM\backslash 0$.
\end{definition}

The most familiar examples of Finsler metric is the Riemannian
metrics, when $F=\sqrt{g_{ij}(x) y^i y^j}$ is a quadratic function of $y$
for any $x$ on the manifold. Similarly as in the Riemannian case, any Finsler metric $F$
gives the length
 for any tangent vector, and this gives arc
length for any piecewise smooth path. We can then define ``distance" as the
minimum of  the arc lengthes among all the piece-wise smooth curves from one point to another \cite{BCS00}.
The distance of a Finsler metric does not satisfy the reversibility
of a metric space,
unless $F$ is absolutely homogeneous,
i.e. $F(x,y)=F(x,-y)$, $\forall x\in M,y\in TM_x$. For simplicity, we will still call it
the distance and denote it as $d(\cdot,\cdot)$.

Among the non-Riemannian examples of Finsler metrics,  Randers
metrics are well-known for its simplicity and importance in geometry and physics.
A Randers
metric $F$ is a sum $F=\alpha+\beta$, where
$\alpha$ is a Riemannian metric and $\beta$ is an one-form whose length
 is everywhere less than $1$.


In \cite{DXP} we have studied the Clifford Wolf translations in Finsler
geometry. We now recall the definitions.
\begin{definition}
A Clifford Wolf translation (or simply a CW-translation) $\rho$ of a Finsler manifold $(M,F)$ is an isometry of $(M, F)$
such that $d(x,\rho(x))$ is a
constant function.
\end{definition}

The interrelation between CW translations and Killing vector fields of
constant lengths in the Riemannian case, due to V. N. Berestovskii and Yu. G. Nikonorov  \cite{BN09}, was generalized to the Finslerian case in
our previous paper \cite{DXP}. We have

\begin{theorem}\label{local-rela}
Let $(M,F)$ be a complete Finsler manifold with positive  injective radius. If $X$ is a Killing vector field of constant
length, then the flow $\phi_t$ generated by $X$ is a CW-translation for
all sufficiently small $t>0$.
\end{theorem}

\begin{theorem}
Let $(M, F)$ be a  compact Finsler manifold. Then there is a $\delta>0$,
such that any CW-translation $\rho$ with $d(x,\rho(x))<\delta$ is
generated by a Killing vector field of constant length.
\end{theorem}

There are some other concepts related to  CW-translations which
has been studied extensively in the Riemannian case. For example, Clifford-Wolf
homogeneous space and restrictively Clifford-Wolf homogeneous space.

\begin{definition}
A Finsler manifold $(M,F)$ is called Clifford-Wolf homogeneous  if
for any two points $x_1,x_2\in M$, there is a CW-translation $\sigma$ such that $\sigma (x_1)=x_2$.
It is called restrictively Clifford-Wolf homogeneous if  for
any point $x\in M$ there is a neighborhood $V$ of $x$, such that for any two points $x_1, x_2\in V$ there is
 a CW-translation $\sigma$ such that $\sigma (x_1)=x_2$.
\end{definition}

We will simply call such a spce CW-homogeneous or restrictively CW-homogeneous.
As we will only deal with compact manifolds in this work, the definition
of restrictively CW-homogeneous can be simplified as the following one.

\begin{definition}
A compact Finsler manifold $(M,F)$ is called restrictively CW-homogeneous if there is a constant $\delta>0$, such that for any pair of points $x$ and $x'$
with $d(x,x')<\delta$ (or equivalently, $d(x',x)<\delta$),
 there is a CW-translation $\rho$ such that $\rho(x)=x'$.
\end{definition}

Obviously the CW-homogeneity or the restrictive CW-homogeneity of a Finsler space $(M, F)$ implies the homogeneity
of the space. Therefore to understand CW-homogeneous Finsler space it is natural to start with  CW-translations of homogeneous Finler spaces. In this case, both the metric data and the conditions for CW-translations can be reduced to
the Lie algebra level, which greatly reduces the complexity of
the problem.

In \cite{DXP} we have studied examples of CW translations on some
compact Lie groups, with left invariant non-Riemannian  Randers metrics. In this work
we will see more examples of CW-translations of homogeneous Randers metrics on
spheres.

\section{Homogeneous  Randers metrics on spheres}

Let $(M,F)$ be a connected compact  Finsler space.
It is called a homogeneous space,
or $F$ is called a homogeneous metric, if its full connected isometry group $G_0=I_0(M,F)$
acts transitively on $M$. It has been proven that $G_0$ is a
compact Lie group \cite{DH02}. Let $H_0\subset G_0$ be the isotropic subgroup of a
 point of $M$. Then the manifold is naturally diffeomorphic to $G_0/H_0$.
In general there are more than one way to express $M$ as a homogeneous space. In fact,
 any connected closed subgroup
$G\subset I_0(M,F)$ which acts transitively on $M$, with the isotropy
subgroup $H\subset G$ fixing the same point,  gives a homogeneous space $G/H$ for $M$.
No matter which $G$ is used, the quotient vector space of the Lie algebras
$\mathfrak{m}=\mathfrak{g}/\mathfrak{h}$
is the same. It can be identified with the tangent space at the chosen point.

The Finsler metric $F$ is totally determined by its restriction to $\mathfrak{m}$,
which is an $Ad_H$-invariant Minkowski norm. This Minkowski norm is transposed to
 other points by the left  translations of $G$. On the other hand, if $K$ is an  effective transitive Lie transformation  group on $M$ and $K_1$ is the isotropy subgroup of $K$ at a fixed point $x\in M$. Then for any $Ad (K_1)$-invariant Minkowski norm on the quotient space $\mathfrak{k}/\mathfrak{k}_1$, one can construct a $K$-invariant Finsler metric on $M$ using the above method.

Let us give an explicit example.  Suppose  $F=\alpha+\beta$ is a homogeneous Randers metric on
$M=G/H$, with $\mathfrak{m}=\mathfrak{g}/\mathfrak{h}$. Then
$\alpha$ is determined by an $\mbox{Ad}_H$-invariant inner product
on $\mathfrak{m}$,
and $\beta$ is determined by an $\mbox{Ad}_H$
invariant element of $\mathfrak{m}^*$. Equivalently, $\beta$ can be
determined by its dual with respect to the inner product, which is an $\mbox{Ad}_H$ invariant
vector $V\in\mathfrak{m}$.

The following lemma  is useful for determining Killing vector fields of constant length, which can generate CW-translations for a homogeneous space, see \cite{DXP}.
\begin{lemma} Let $(M, F)$ be a homogeneous Finsler space and $G$ be its full group of isometries, Lie\,$G=\mathfrak{g}$.
Then a Killing vector field generated by $X\in \mathfrak{g}$
is of constant length $1$
if and only if the projection of the $Ad_{G}$-orbit
of $X$ to $\mathfrak{m}$ is contained in the indicatrix.
\end{lemma}

By studying the projections of the orbits, we can find the wanted homogeneous
Finsler metric from its indicatrix at the chosen point.

Now we turn to the main subject of this paper, namely, spheres with homogeneous Randers metrics.
Suppose $G$ is an effective transitive Lie transformation group of $S^n$  and $H$ is the isotropy subgroup of $G$ at a fixed point. If $F=\alpha+\beta$ is
a $G$-invariant Randers metric on $S^n$,  then so is $\alpha$ (\cite{D08}).
A complete list of Lie groups which admit an effective transitive action on $S^n$ was obtained by Montgomery and Samelson (\cite{MS43}). The list results in the following:
\begin{lemma}\label{homospacelistforsphere} The following list of Riemannian homogeneous spaces
$G/H$ for spheres is complete, in any case $G$ is a connected subgroup of the full isometry group of
a $G$-invariant Riemannian metric $\alpha$ on $S^n$.

(1) $S^n=SO(n+1)/SO(n)$, $n\geq 1$,

(2) $S^{2n+1}=SU(n+1)/SU(n)$, $n\geq 1$,

(3) $S^{2n+1}=U(n+1)/U(n)$, $n\geq 1$,

(4) $S^{4n+3}=Sp(n+1)/Sp(n)$, $n\geq 1$,

(5) $S^{4n+3}=Sp(n+1)U(1)/Sp(n)U(1)$, $n\geq 1$,

(6) $S^{4n+3}=Sp(n+1)Sp(1)/Sp(n)Sp(1)$, $n\geq 1$,

(7) $S^6=G_2/SU(3)$,

(8) $S^7=Spin(7)/G_2$,

(9) $S^{16}=Spin(9)/Spin(7)$.
\end{lemma}

The list gives all the possible $G\subset I_0(M,F)$ which acts transitively on  spheres, for all possible homogeneous Finsler metrics $F$ on them.
In the special case of Randers metrics, to produce a non-Riemannian metric on $G/H$, we must have a non-zero vector in $\mathfrak{m}$ which is fixed under $Ad (h)$, for any $h$ in the isotropy subgroup. This is equivalent to the condition that the isotropy representation of $H$ on $\mathfrak{m}$ has a non-zero trivial subrepresentation. From the above list it is obvious that  this is the case only in  (2)-(5).
Therefore we only need to deal with the cases of (2)-(5).

In (2) and (3), where $G=U(n)$ or $SU(n)$, the full isometry
group of any $G$-invariant non-Riemannian Randers metric must be $U(n)$.
We will study this case in  more detail  in Section \ref{u-isometry}.
In (4) and (5), the full isometry group can be $Sp(n)$, $Sp(n)U(1)$ or $U(2n)$, in this case our main focus will be on the group
$Sp(n)U(1)$.

\section{CW-translations of  left invariant Randers metrics on $SU(2)$}
This work is motivated by the particular example $S^3$ which can be
regarded as $SU(2)=SU(2)/SU(1)$, $Sp(1)=Sp(1)/Sp(0)$,
 $U(2)/U(1)$ or $Spin(1)U(1)/U(1)$  appearing in
each case of (2)-(5) in Lemma \ref{homospacelistforsphere}.

Let $F$  be a non-Riemannian homogeneous Randers metric on $S^3=SU(2)/{e}$. There is no
Killing vector field of non-zero constant length generated by the elements of
$\mathfrak{g}=\mathfrak{su}(2)$ (see \cite{DXP}).
Now let us see if we can find a Killing vector field of constant length from
the Lie algebra of the full connected isometry group
$U(2)=Sp(1)U(1)$. In $U(2)$, the center vectors generate
some special Killing vector fields. These Killing vector fields generate CW-translations of the symmetric
metric on $S^3$. Moreover,  they have constant length with respect to any $U(2)$-invariant Finsler metric on $S^3$.
This case is uninteresting and we just ignore them. So let us try to find those Killing vector fields of constant length generated by non-central elements of
$\mathfrak{u}(2)$.

Denoting $\mathfrak{g}=\mathfrak{u}(2)$ and $\mathfrak{h}=\mathbb{R}$ the Lie algebras of $G$ and $H$ respectively, we have a Lie algebra decomposition $\mathfrak{g}=\mathfrak{u}(2)=\mathfrak{su}(2)\oplus
\mathbb{R}$. The subalgebra $\mathfrak{h}$ is generated by an element of the form
$(V,1)\in\mathfrak{g}$ with $V\neq 0$.
On $\mathfrak{u}(2)$, there is a standard inner product, i.e., $\langle A,B\rangle_{eq}=-trAB$. The above decomposition
of $\mathfrak{u}(2)$ is orthogonal with respect to this inner product.  Moreover, the  restriction of this inner product to $\mathfrak{m}\cong \mathfrak{su}(2)$
induces the the standard Riemannian metric on $S^3$.
For any
$X$ in $\mathfrak{su}(2)$ with $|X|>|V|$,  the ${\rm Ad}_G$-orbit
of $(X,1)$ is a  $2$-dimensional round sphere centered at $(0,1)$ with respect to  $\langle\,,\,\rangle_{eq}$.
The projection of this orbit to
$\mathfrak{m}=\mathfrak{g}/\mathfrak{h}=\mathfrak{su}(2)$ is a sphere of
the same radius, with  center  shifted to $-V$. By the assumption, this  sphere
still surrounds the origin. This observation gives a method  to
find  homogeneous Randers metrics with the prescribed indicatrix, such that
$(X,1)$ generates a Killing vector field of constant length. In fact,
up to a constant scalar, $\beta$ is the
dual of $V$ with respect to $\langle\,,\,\rangle_{eq}$, $\alpha$ has an
ellipsoid indicatrix, which is  a round sphere with respect to the metric $langle\,,\,\rangle_{eq}$, with center stretched in
the direction of $V$ (?). Both $\alpha$ and $\beta$ are $Ad_G$-invariant(?), so they are
 $Ad_{V}$-action as well.


Once we have found a non-vanishing Killing vector  field of constant length of
a homogeneous Randers metric $F$,  we can find an
$Ad_G$-orbit of Killing vector fields of the same constant length. It is easily seen that these
Killing vector fields exhaust all the tangent directions at the origin. By Theorem
\ref{local-rela},  the homogeneous non-Riemannian Randers metrics
 constructed above on $S^3$ are restrictively CW-homogeneous.

As a by-product, this construction can be generalized to other connected
compact Lie groups.

\begin{proposition}
On any connected compact Lie group $G$, there is a  left invariant non-Riemannian
Randers metric $F$ which makes $(M, F)$ a   restrictively CW-homogeneous Finsler space.
\end{proposition}
\begin{proof}
Let $G'=G\times S^1$, whose Lie  algebras is
$\mathfrak{g}'=\mathfrak{g}\oplus \mathbb{R}$. Select a nonzero $V$ in ${\mathfrak g}$ such that the one-parameter subgroup $\exp t(V)$ is isomorphic to $S^1$. Let $H'$ be the subgroup of $G'$ generated by
$(V,1)$. Obviously $H'\cong S^1$ and $G'/H'$ is a homogeneous space for
$G$. Choose any bi-invariant metric on $G$ and denote by  $\langle \,,\,\rangle_{eq}$
the inner production  induced on $\mathfrak{g}$.
We can assume $\langle V,V\rangle_{eq}<1$. The sphere
$S=\{(X,1)|\langle X,X\rangle_{eq}^{1/2}=1\}$ is the union of some ${\rm Ad}_{G'}$-orbits.
Projected to $\mathfrak{m}\cong\mathfrak{g}$, its center is
shifted to $-V$. Using the above argument one can similarly find a $G'$-invariant   Randers metric  on $G\simeq G'/H'$ with
 the above sphere in $\mathfrak{g}$ as the indicatrix. Then any
vector $(X,1)\in S$ generates a Killing vector field of constant length $1$, and
 these vectors exhaust all the tangent directions. Therefore this Randers metric
makes $G$ a restrictively CW-homogeneous Finsler spaces.
\end{proof}

For general compact Lie groups, it is still unknown whether the word ``restrictively"
can be removed. But for the special case
$SU(2)$, the answer is positive.

\begin{proposition}
On $S^3$, there are non-Riemannian homogeneous Randers metrics which
makes it CW-homogeneous.
\end{proposition}
\begin{proof}
 The homogeneous Randers metric we choose is the one constructed as above.
The proof is carried out by a closer observation of the geometry of
the Randers metric and Killing vector  fields in the construction. Choose a
bi-invariant inner product on $\mathfrak{su}(2)$ such that the induced
metric makes $SU(2)$ the standard unit sphere.  Without losing generality,
we can assume  that $X$ has length $1$ and $V$ has length $l<1$ with respect to this metric. Suppose $(X,1)$ generates the Killing vector field (of constant length)
of the Randers metric $F$. The flow of isometries generated by
$(X,1)$ are $\phi_t(g)=\mbox{exp}(tX)g\mbox{exp}(-tV)$ which gives a
geodesic at $g$. Notice that $\exp(\pi X)=-\mbox{id}$ for each $X$ with
length $1$, since the length of $X$ is 1 implies that the eigenvalues of $X$
are $\pm\sqrt{-1}$. Thus $\exp(\pi X)$ is a unitary conjugation
of $\exp (\mbox{diag}(\pi\sqrt{-1},-\pi\sqrt{-1}))=-\mbox{id}$.

Since these $X$ can exhaust all the unit vectors of $\mathfrak{su}(2)$, the
geodesics in all directions starting from $g$ with $t=0$ will end at
$-g\mbox{exp}(\pi V)$ with $t=\pi$.  This means that all those geodesics
from $g$ to $-g\exp (\pi V)$ have the same length $\pi$.

Any point can be
reached by a geodesic from $g$ for $t\in [0,\pi]$.
Otherwise we can choose a shortest geodesic from $g$ to it, passing $-g\mbox{exp}(\pi V)$
in the midway. Then the geodesic from $g$ to $-g\mbox{exp}(\pi V)$ can be changed to another one
which turns a angle at $-g\mbox{exp}(\pi V)$, and the new path is still a shortest path.
This is a contradiction because the new path is not a smooth geodesic.(?)

Those
geodesics do not intersect each other when $t$ is restricted to
$(0,\pi)$. Otherwise, there will be a pair of unit vectors $X_1$ and $X_2$ in
$\mathfrak{su}(2)$, and $t_1,t_2\in(0,\pi)$, such that
\begin{equation}
\exp (t_1 X_1)g\exp (-t_1 V)=\exp (t_2
X_2)g\exp (-t_2 V).
\end{equation}
If $t_1=t_2$, then we have $X_1={\rm Ad}_g X_2$. Thus $X_1=X_2$.

If $t_1\neq t_2$,  say $t_2< t_1$, then we have
\begin{equation}
\exp(t_1 X_1)=\exp(t_2 X_2)\exp((t_1-t_2){\rm Ad}_g V).
\end{equation}
The left side gives a point on a  geodesic sphere with radius $t_1$
centered at the point $e_0$ representing the identity matrix. The
right side gives a point on a geodesic sphere with radius $(t_1-t_2)l$
which is centered at the point $\exp(t_2 X_2)$ on  geodesic sphere centered at $e_0$ with  radius
$t_2$. Therefore there is  a path
from $e_0$ to the point given by the left side that has a length
smaller than $t_1$, which is a contradiction.

So all those geodesic flow curves from $g$ to $-g\mbox{exp}(\pi V)$
are the shortest ones. This property only depends on the length of
the $X$'s and the length of $V$. Change of $g$ only results in a change
of unitary conjugation, without changing the lengthes. Therefore given any two points $g_1$, $g_2$ in $SU(2)$, one can find a CW-translation $\phi_t$, with $t\in[0,\pi]$,
where $\phi_t$ is the flow of a Killing vector field generated by  certain  $(X,1)$, such that $\phi_t(g_1)=g_2$. This completes the proof of the proposition.
\end{proof}

\section{Randers Spheres with unitary isometry groups}\label{u-isometry}
\label{s2n+1-section}
Let $F=\alpha+\beta$ be a non-Riemannian homogeneous Randers metric on $S^{2n+1}$,
such that $I_0(S^{2n+1},F)=U(n+1)$. Then the sphere can be presented as $G/H$, with
$G=U(n)$ and $H=U(n-1)\subset G$.
 Their Lie algebras are $\mathfrak{g}=\mathfrak{u}(n)$ and
$\mathfrak{h}=\mathfrak{u}(n-1)$ respectively. The tangent space at the origin is the quotient space
$\mathfrak{m}=\mathfrak{m}_0\oplus\mathfrak{m}_1$, where
$\mathfrak{m}_0=\mathbb{R}$ and
$\mathfrak{m}_1=\mathbb{C}^n$. The isotropy subgroup acts trivially on $ \mathfrak{m}_0$  and acts on $ \mathfrak{m}_1$ by left multiplication.
The projection of $X\in \mathfrak{g}$ to $\mathfrak{m}$ is equal to $X(0,\ldots,0,\sqrt{-1})^*$.

Since the underlying Reiamnnian metric  $\alpha$ is also invariant under $G$, it induces an
${\rm Ad}_H$-invariant linear metric (still denoted by $\alpha$) on $\mathfrak{m}$, which must have
the form $\alpha^2(q,u)=a|q|^2+bu^* u$,
$\forall q\in \mathbb{R}$ and $u\in\mathbb{C}^n$, with positive constant $a$ and $b$.
The standard inner product, i.e., the one with $a=b=1$,
is induced by the symmetric standard Riemannian
metric. The corresponding inner product on $\mathfrak{m}$ will be denoted by $\langle\,,\,\rangle_{eq}$.

The non-vanishing $1$-form $\beta$ is also $G$-invariant, so
it is induced by an $\mbox{\rm Ad}(H)$-invariant vector $V \in {\mathfrak m}$. This means that $V$ must be contained in ${\mathfrak m}_0$.  Thus $V$ has the form
$\mbox{\rm diag}(0,\ldots,0,c)^T$, with $|c|<\sqrt{a}$.

Suppose there is a vector $X\in \mathfrak{u}(n+1)$ which generates a Killing vector field of
constant length $L>0$ with respect to $F=\alpha+\beta$. For simplicity, we  assume that $X$ is not
in the center of $\mathfrak{u}(n+1)$,
since any vector in the center of $\mathfrak{u}(n+1)$ generates a Killing vector field of
constant length for any homogeneous Finsler metric on $S^{2n+1}=U(n+1)/U(n)$.
Moreover, it is also a CW-translation of the CW-homogeneous  Riemannian metric on $S^{2n+1}$ (i.e., the standard metric).

Up to a unitary conjugation, we can assume $X$ to be
diagonal. By the action of the Weyl group, each eigenvalue of $X$ can
appear at the down right corner. The projection of those diagonal matrices
in the orbit of $X$ to
$\mathfrak{m}$ can be denoted as $(0,\ldots,0,a_i)$, in which
$a_1\sqrt{-1},\ldots,a_{n+1}\sqrt{-1}$ are all the eigenvalues of
$X$. Then these $a_i$'s must be the solution of the equation
\begin{equation}
\sqrt{a}|x|+cx=L.
\label{eq}
\end{equation}
By the assumption, $X$ has at least two distinct eigenvalues.  Then from \eqref{eq} it is easily seen that  $X$ has
 exactly two distinct eigenvalues with opposite signs. So by a suitable unitary conjugation, we can
assume that $X=\sqrt{-1}(x_1
\mbox{I}+x_2\mbox{diag}(-m\mbox{I}_l,l\mbox{I}_m))$, where $l$
and $m$ are natural numbers satisfying $l+m=n+1$, $x_2\neq 0$ and
\begin{equation}
(x_1-mx_2)(x_1+lx_2)<0.\label{4}
\end{equation}
Suppose $U\in U(n+1)$. Denote   its last row
denoted by $(u^*,v^*)$. Then the value of the metric $\alpha^2$ at the projection
of $UXU^*$ in $\mathfrak{m}$ is
\begin{eqnarray}
b x_2^2(m^2(|u|^2-|u|^4))+l^2(|v|^2-|v|^4)+2ml|u|^2|v|^2)\nonumber \\
+a (x_2(-m|u|^2+l|v|^2)+x_1)^2.
\end{eqnarray}
Denote $t=l|v|^2-m|u|^2\in [-m,l]$. Since
$|u|^2+|v|^2=1$, we have
\begin{equation}
|u|^2=\frac{1-t}{n+1},\\
|v|^2=\frac{t+m}{n+1}.\\
\end{equation}
Then the  above $\alpha^2$ term can be summarized as
\begin{equation}
f(t)=(a-b)x_2^2 t^2+[(l-m)x_2^2 b+2a x_1 x_2]t+(x_2^2 b
ml+a x_1^2).
\end{equation}

The $\mathfrak{m}_0$-coordinate of the projection of $UXU^*$ is $x_2 t+x_1$, so
its $\beta$ value is $c(x_2 t+x_1)$.

The  Killing vector field generated by $X$ having a constant length $L>0$ is equivalent
to the equation
\begin{equation}
f(t)\equiv (-c(x_2 t+x_1)+L)^2, \forall t\in [-m,l].
\end{equation}

For any $m$,$l$, $x_1$, $x_2$ and $L$, we can uniquely solve the triple $(a,b,c)$,
\begin{eqnarray}
a&=&b+c^2,\\
\label{a}
b&=&L^2[(\frac{(n+1) x_2}{2})^2-(\frac{l-m}{2}x_2+x_1)^2]^{-1},\\
\label{b}
c&=&-\frac{b}{L}(\frac{l-m}{2}x_2+x_1).
\label{c}
\end{eqnarray}

To summarize, we have proved the following theorem.
\begin{theorem}\label{10}
For any $X\in u(n+1)$ with exactly two eigenvalues of different signs and different absolute values, there is a
non-Riemannian homogeneous Randers metric on $S^{2n+1}=U(n+1)/U(n)$,
which is unique up to a scalar, such that $X$  generates a CW-translation.
\end{theorem}

This theorem can also be stated as the following.
\begin{theorem}\label{5}
Let  $F$ be a non-Riemannian homogeneous Randers metric
  on $S^{2n+1}=U(n+1)/U(n)$,
determined by the parameters $(a,b,c)$  defined by \eqref{a},\eqref{b} and \eqref{c}. Then there exists a  non-vanishing
 Killing vector field of constant length which is not in the center of ${\mathfrak u}(n+1)$ if and
only if $a=b+c^2$. The Killing vector fields of constant length $L>0$ generated by elements of the
center of $\mathfrak{u}(n+1)$ are in one-to-one correspondence with the unitary matrices with
 exactly  two eigenvalues $(-\frac{Lc}{b}\pm\sqrt{\frac{L^2}{b}+\frac{L^2 c^2}{b^2}})\sqrt{-1}$.
\end{theorem}

It should be noted that the Riemannian CW-translations of the symmetric space $S^{2n+1}$
can also be derived from the above discussion. In fact,  when $c=0$, we get those $X\in u(2n+1)$
whose eigenvalues have the same absolute values. They are all the Killing vector fields of constant
length of the symmetric sphere, which commute with the given one, i.e.,
$\sqrt{-1}I\in \mathfrak{u}(2n+1)$. The matrix $X$ in Theorem \ref{10} and Theorem \ref{5} can be written as
\begin{equation}
X=\sqrt{-1}(x_1+\frac{l-m}{2}x_2)I+\sqrt{-1}\frac{l+m}{2}x_2\mbox{diag}(-I_l,I_m).
\end{equation}
This means that   there is a natural correspondence between
the Killing vector fields of constant length of a homogeneous non-Riemannian Randers
metric and the pairs of Killing vector fields of constant length of a symmetric metric, with the later pair commuting
with each other and having different lengthes. Accordingly we have the correspondence for CW-translations.

The condition $a=b+c^2$ implies that the indicatrix of the Randers metric
is a sphere in $\mathfrak{m}$ with respect to the inner product $\langle\cdot,\cdot\rangle_{eq}$ (in general not centered at $0$). Therefore the projection of the $\mbox{Ad}_G$-orbit of $X$ is contained in a sphere.
In fact it is projected onto that sphere. If we choose a unitary
matrix $U$ so that its last row is $(u^*,v^*)$, with $u\in \mathbb{C}^l,v\in\mathbb{C}^m$
satisfying
$v^*=(0,\ldots,0,s)$, $s\in [-1,1]$,  and its last column is of the form
$(\sqrt{1-s^2}w^*,s)^*$, where $w$ is a unit vector in
$\mathbb{C}^n$, then for any $s\in [-1,1]$, and any
unit vector $w$, this $U$ can be found by the process of choosing a unitary basis. First
choose $(\sqrt{1-s^2}w^*,s)$, then choose the next $m-1$ vectors from
$(\sqrt{1-s^2}w^*,s)^{\perp}\cap(0,\ldots,0,1)^{\perp}$, then the others, and rearrange the
order at the end.(?)

When calculating $UXU^*$, we only need to consider the cases of
$s\in [0,1]$. The last column of
\begin{equation}
UXU^*=\sqrt{-1}U(x_1 I+x_2\mbox{diag}(-m\mbox{I}_l,l\mbox{I}_m))U^*
\end{equation}
is $\sqrt{-1}((n+1)x_2 \sqrt{1-s^2}s w^*, (n+1)x_2 s^2-mx_2+x_1)^*$.
They can give all the points on the sphere for
$\langle \cdot,\cdot\rangle_{eq}$, which is centered at $(0,(l-m)x_2/2+x_1)$
with radius $(n+1)|x_2|/2$.
So it gives all directions in $\mathfrak{m}$. The resulting non-Riemannian
homogeneous Randers metrics on $S^{2n+1}$ is
restrictively CW-homogeneous. We Shall prove in the following that they are in fact CW-homogeneous.

If $a\neq b+c^2$,
then the only Killing vector fields of constant lengths of the homogeneous Randers metric constructed by the triple are in the center  ${\mathfrak u}(1)$. They only gives a flow of CW-translations on the sphere. Obviously they are not restrictively Clifford-Wolf homogeneous.

\begin{theorem}
On $S^{2n+1}=U(n+1)/U(n)$, any non-Riemannian homogeneous
Randers metrics determined by a triple $(a,b,c)$ with $a=b+c^2$,
are Clifford-Wolf homogeneous.
\end{theorem}
\begin{proof}
 Up to a  scalar multiple, we can write $X$ as
$X=\sqrt{-1}(x\mbox{I}+\mbox{diag}(-I_l,I_m))$, with $0<|x|<1$. With respect to
the homogeneous Randers metrics constructed above, any $X'$ in the
$\mbox{Ad}_{U(2m)}$-orbit of $X$ generates a flow of isometries
$\phi_t(v)=\exp (tX')v$, $v\in S^{2n+1}\in \mathbf{C}^{2n+2}$ on
the sphere. Up to a scalar constant, $t$ parameterizes the length of
geodesic flow curves. For any $X'$ in the same orbit, the flow
curves $\phi_t(v)$ gives all the geodesic starting at $v$ with
$t=0$, and all reach $-\mbox{exp}(x\pi\sqrt{-1})v$ when $t=\pi$.
For any other $v'$ on the sphere, the shortest geodesic from $v$ to
$v'$ must reach $v'$ when $t\leq \pi$. Otherwise we can change it by choosing
another geodesic for the segment $t\in [0,\pi]$, then the path is not geodesic
but gives the same shortest distance from $v$ to $v'$. This is a contradiction.

The study on those geodesics with $t\in (0,\pi)$
needs the following lemma about the
estimate of the eigenvalues of the product of two unitary matrices, whose  proof will be provided
in the appendix.
\begin{lemma}\label{eva-eigenvalue}
Suppose  $P$ and $Q$ are two unitary matrix in $U(n)$, and denote the eigenvalues of
$P$  as $e^{a_i\sqrt{-1}}$, $a_i\in (-\pi,\pi)$,  the
eigenvalues of $Q$  as $e^{b_i\sqrt{-1}}$, $b_i\in
(-\pi,\pi)$, $i=1,2,\ldots,n$. Let $m_1$ and $m_2$ be the maximum
and minimum of the $b_i$'s, respectively. Then the eigenvalues of $PQ$  must be of the form
 $e^{c_i\sqrt{-1}}$, with $c_i\in [a_i+m,a_i+M]$.
\end{lemma}

If any two of those geodesics intersect within $t\in (0,\pi)$, i.e.,
there are $X_1$ and $X_2$ which are unitary conjugate to $X$, and
$t_1$ and $t_2$ in $(0,\pi)$, such that $\exp (t_1
X_1)v=\exp (t_2 X_2)v$, then $\exp(t_1X_1)\cdot
\exp(-t_2 X_2)$ has an eigenvalue $1$. If $t_1\neq t_2$, then we
may assume $t_1>t_2$. The eigenvalues of $\mbox{exp}(-t_2
X_2)$ have the form $e^{b_i\sqrt{-1}}$, $b_i\in [-t_2 x - t_2, -t_2
x + t_2]$. By the lemma, the eigenvalues of $\exp (t_1
X_1)v=\exp (t_2 X_2)v$ have the form $e^{c_i \sqrt{-1}}$,
$c_i\in [t_1 x + t_1 - t_2 x - t_2, t_1 x + t_1 - t_2 x + t_2]$ or
$c_i\in [t_1 x - t_1 - t_2 x - t_2, t_1 x - t_1 - t_2 x + t_2]$. In
both cases $c_i\in (-2\pi,0)\cup(0,2\pi)$, and $1$ can not be an
eigenvalue of the product. So if any two geodesics starting from
$v$ with $t=0$ intersect in the midway, then they have the same length
between the two common points. Therefore it suffices to
prove that all the geodesics from $v$ to $-\mbox{exp}(x\pi\sqrt{-1})v$
are the shortest pathes. In fact, when $t=t_1=t_2\in (0,\pi)$,
$\exp(-tX_1)\exp(tX_2)v=v$, we can prove that these two
geodesic coincides for all $t\in[0,\pi]$. The essential steps are
left in the Appendix. The flows generated by the Killing vector
fields in the orbit of $X$ are CW-translations for $t\in (0,\pi]$.
This completes the proof of the theorem.
\end{proof}

\section{Randers metrics on $S^{4n+3}$ with $I_0(S^{4n+3},F)\subset Sp(n+1)U(1)$ }

Now we start the discussion on the cases (4) and (5) in the list of
homogeneous spaces for spheres. Let $F$ be a non-Riemannian
homogeneous Randers metric on $S^{4n+3}$, $n>0$, such that its connected isometry group
contains $Sp(n+1)$. Then $I_0(M,F)$ must be among $U(2n+2)$, $Sp(n+1)U(1)$ or
$Sp(n+1)$. The unitary case has already been discussed. So we only need to consider
the case $I_0(M,F)=Sp(n+1)U(1)$ or $I_0(M,F)=Sp(n+1)$.

For $G=Sp(n+1)$ or $Sp(n+1)U(1)$,
$\mathfrak{m}$ can be decomposed as $\mathfrak{m}_0\oplus\mathfrak{m}_1$,
where $\mathfrak{m}_0=\mbox{Im}\,\mathbb{H}$ is the $3$-dimensional trivial representation of $Sp(n-1)$,
and $\mathfrak{m}_1=\mathbb{H}^{n}$ with the action of $Sp(n)$ by
left multiplication.
When regarded as a $Sp(n)U(1)$ representation, $\mathfrak{m}_1$ also has the action of
$U(1)$-scalar multiplication from the right, and $\mathfrak{m}_0$ is further decomposed
into the sum of the $1$-dimensional trivial representation of $U(1)$, generated by $\mathbf{i}
\in\mathbb{H}$ (which is identified with $\sqrt{-1}\in U(1)$), and the 2-dimensional space spanned
by $\mathbf{j}$ and $\mathbf{k}$, on which $U(1)$ acts as the rotation group.
The projection from the Lie algebra of
$G=Sp(n+1)$ or $Sp(n+1)U(1)$
to $\mathfrak{m}$ is just the differentiation of the group action
on $(0,\ldots,0,1)^*\in \mathfrak{m}\subset\mathbb{H}^{n+1}$ at $I$.

We have a standard inner product on $\mathfrak{m}$ induced by a symmetric
Riemannian metric on the sphere. It will be denoted as $\langle \,,\,\rangle_{eq}$.
Any $\mbox{Ad}_{Sp(n)Sp(1)}$-invariant linear metric on $\mathfrak{m}$ can be written
as $\alpha^2(u,q)=\mbox{Re}(a q^* q + b u^* u)$, $q\in \mbox{Im}\,\mathbb{H}$, $u\in\mathbb{H}^n$.
The standard inner product $\langle\,,\,\rangle_{eq}$ is corresponding to the case $a=b=1$.
A non-Riemannian Randers metric $F$ can written as
\begin{equation}\label{alpha+beta}
F=\alpha+\langle\cdot,V\rangle_{eq},
\end{equation}
where $V\in \mathfrak{m}_0$ can be any non-zero vector if the isometry
group is $Sp(n+1)$, or generated by $\mathbf{i}$ if the isometry group
is $Sp(n+1)U(1)$.
We now prove
\begin{proposition}\label{nonzero}
If $X\in \mathfrak{sp}(n+1)$ generates  a Killing vector field of constant length with respect to a
non-Riemannian homogeneous Randers metric $F$ on $S^{4n+3}$, then
$X=0$.
\end{proposition}
\begin{proof}
Up to a $Sp(n+1)$ conjugation, one can assume that  $X$  is a
diagonal matrix in $\mathfrak{gl}(n+1,\mathbb{H})$. If a diagonal entry of $X$ is not $0$, say $q\in {\mathbb H}$, then there is an element $\sigma$ in the
Weyl group such that the
 down right corner of $\sigma (X)$ is $q$.  Moreover, up to a $Sp(n+1)$ conjugation,  we can further assume that the down right corner of $\sigma (X)$
 is real proportional to $V$ in (\ref{alpha+beta}). There are two choices for this down right corner, namely $\frac{|q|}{|V|}V $ and $-\frac{|q|}{|V|} V$.
The projections of the corresponding matrix to $\mathfrak{m}$ form an opposite pair
of vectors, denoted by $q_1$ and $q_2$. Then we have $\alpha (q_1)=\alpha (q_2)$, $F(q_1)=F(q_2)$ and $\beta (q_1)=-\beta(q_2)$. This means
that $\beta (q_1)=\beta (q_2)=0$. Therefore we have  $\langle V,V\rangle_{eq}=0$. This is
a contradiction with $V\neq 0$.
\end{proof}

From now on, we assume that $G=I_0(S^{4n+3},F)=Sp(n+1)U(1)$.
The vector $V$ in (\ref{alpha+beta}) will be denoted as
$(0,\ldots,0,c\mathbf{i})^T$, where $c$ is a nonzero real number.
The metric $\alpha$ induced on $\mathfrak{m}$ can be written
as $\alpha^2=\mbox{Re}(a_1\lambda_1^2+a_2\lambda_2^2+a_2\lambda_3^2+b u^* u)$,
$q=\lambda_1\mathbf{i}+\lambda_2\mathbf{j}+\lambda_3\mathbf{k}$, $u\in\mathbb{H}^n$,
for some positive parameter $a_1$, $a_2$ and $b$. Moreover,  we  have $a_2\neq b$, since otherwise
$I_0(S^{4n+3},F)$ will be $U(2n+2)$.

As in the last section, the center of $G$ generates CW-translations for any
homogeneous Finsler metric on the sphere, and they are not the ones we are searching for.
Assume there is a non-central vector in $\mathfrak{g}$, which
generates a Killing vector field of constant length with respect to a non-Riemannian homogeneous Randers metric
$F$. Then it can be written as
 $(X,x)\in \mathfrak{sp}(n+1)\oplus\mathbb{R}$, with $X\neq 0$.
Then  Proposition \ref{nonzero} asserts that $x$ is nonzero either.

Up to a suitable $Sp(n+1)$ conjugation, we can assume
\begin{equation}
X=\mbox{diag}(x_1\mathbf{i},\ldots,x_{n+1}\mathbf{i}),
x_i\in\mathbf{R},  i=1,\ldots,n+1.
\end{equation}
Using the action of the Weyl group,
we can reorder the $x_i$'s freely, and change $x_{n+1}$ to
$-x_{n+1}$. In this way, we get a set of elements in $\mathfrak{g}$. Their  projections to $\mathfrak{m}$ have the form $(0,\ldots,0,(\pm x_i+x)
\mathbf{i})^T$, $i=1,\ldots,n+1$. They all have the same $F$ values.
Similar discussions shows that $\{\pm x_i+x, i=1,\ldots,n+1\}$
must take exactly $2$ values with opposite signs. So the $|x_i|$'s must
be equal to each other. Using actions of the Weyl group, we can change all the $x_i$'s
to the same positive number.
Then we can write $X=x'\mathbf{i}I\in sp(n+1)\subset gl(n+1,\mathbb{H})$, with $x'>|x|>0$.

Now we need to calculate the projection of the $\mbox{Ad}_{Sp(n+1)U(1)}$-orbit, which is also
the $\mbox{Ad}_{Sp(n+1)}$-orbit of $(X,x)$. Suppose
$Q=Q_1+Q_2\mathbf{j}\in Sp(n+1)$, where $Q_1$ and $Q_2$ are complex matrices, and $\sqrt{-1}$
is identified with $\mathbf{i}$.
The condition $Q\in Sp(n+1)$ implies
\begin{eqnarray}
Q_1^* Q_2 - Q_2^T \bar{Q_1}&=&0,\\
Q_1^* Q_1 + Q_2^T \bar{Q_2} &=& I.
\end{eqnarray}
Then $Q^* X Q = -\mathbf{i}I + 2 Q_1^* Q_1 \mathbf{i} + 2 Q_1^* Q_2 \mathbf{k}
= \mathbf{i}(-I+2Q_1^*(Q_1+Q_2\mathbf{j}))$.
As we will project it to $\mathfrak{m}$, we only need to see its last column.

The last row (?) of $Q$ can be denoted as $(\sqrt{1-|q|^2}w,q)$,
where $w^*\in\mathbb{H}^{n-1}$
is a unit vector, and $q=q_1+q_2\mathbf{j}$, $q_i\in\mathbb{C}$.
We first assume that the last row of $Q_1^*$ is $(0,\ldots,0,\bar{q_1})$. Then the projection
of $(Q^* X Q,x)$ in $\mathfrak{m}$ is
\begin{equation}
-((2x'(1-|q_1|^2-|q_2|^2))\mathbf{i}q_1 w,(x'(2|q_1|^2-1)+x)\mathbf{i}+
2x'\bar{q_1}q_2 \mathbf{k})^*.
\end{equation}
It gives all the points on the sphere $S$ for $\langle \,,\,\rangle_{eq}$, which is centered at $x$
and has a radius $x'$ if we can find the suitable $w$ and $q$. Note that this is the same sphere appearing in the last section with $m=l$.
 With $Sp(n)$ changed by
$SU(2n)$ or $U(2n)$, the $\mbox{Ad}_{U(2n)}$-orbit ( containing the $\mbox{Ad}_{Sp(n)}$-orbit )
will be mapped onto $S$. If the $\mbox{Ad}_{Sp(n)}$-orbit is also mapped onto
$S$, then the Randers metric must be the one
constructed in the last section, and in this case we have  $I(S^{4n+3},F)=U(2n+2)$.

To see the $\mbox{Ad}_{Sp(n)}$-orbit is also mapped onto $S$, we  need the following lemma.
\begin{lemma}
For any unit vector in $\mathbb{H}^n$, there is $Q\in Sp(n)$, such that the last row of
$Q$ is the given vector, and the last row of $Q^*$ has the form $(q',0\ldots,0,q)$, where the imaginary part of $q'$
is contained in the real span of   $\mathbf{j}$ and $\mathbf{k}$.
\end{lemma}

The proof can be given by the process
of choosing an orthogonal basis for the inner product $\langle x,y\rangle=x^* y\in\mathbb{H}$. First
we  choose an arbitrary unit vector for the last row of $Q$, then choose the next $n-2$ in
orthogonal complement of the first one, at the same time perpendicular to $(0,\ldots,0,1)$.
For the last basis vector, we can use a suitable unit scalar multiplication to make its
last term only contain $\mathbf{j}$ and $\mathbf{k}$. (?)

We conclude this paper with the following theorem.
\begin{theorem}
Let $F$ be a homogeneous Randers metric on $S^{4n+3}$ with  $I_0(S^{4n+3},F)=Sp(n+1)U(1)$. Then  any Killing vector
field of constant length for $F$ is generated by a central vector in the Lie algebra of the
isometry group.
\end{theorem}

\section{Appendix: estimates of the eigenvalues of unitary matrices}
In this section we give a proof of Lemma \ref{eva-eigenvalue}.

Up to a suitable unitary conjugation, we can assume
$Q=\exp B$, where
$B=\mbox{diag}(b_1\sqrt{-1},\ldots,b_n\sqrt{-1})$, $b_i\in
(-\pi,\pi)$ for $i=1,\ldots,n$.

First we consider a special case such that $P \exp (tB)$ has
no multiple eigenvalues for each $t\in [0,1]$. The eigenvalues
of $P \mbox{exp}(tB)$ can then be presented as smooth functions
$\lambda_1(t), \ldots,\lambda_n(t)$ of $t\in [0,1]$. We can also
find the corresponding unit eigenvectors $v_1(t),\ldots,v_n(t)$, respectively, which
are vector-valued smooth  function of $t$.  Differentiating the equation
$P\exp (tB)v_1(t)=\lambda_1(t) v_1(t)$ with respect to $t$ and
taking $t=0$,  we get
\begin{equation}
Pv'_1(0)+Bv_1(0)=\lambda'_1(0) v_1(0) + \lambda_1(0)
v'_1(0).\label{3}
\end{equation}
Taking the inner product with $v_1(0)$ for both sides of (\ref{3}),
and noticing that $Pv'_1(0)$ and $v'_1(0)$ are orthogonal to $v_1(0)$,
we have $Bv_1(0)=\lambda'_1(0)$. So $\lambda'_1(0)/\sqrt{-1}$ is
bounded between the minimum and the maximum of all the $b_i$'s.

This calculation is valid for all $i=1,\ldots,n$ and all $t\in
[0,1]$ with $P$ replaced by $P\exp (tB)$. If we write
$\lambda_i(t)=\mbox{exp}(c_i(t)\sqrt{-1})$ with smooth $c_i(t)$
satisfying $c_i(0)=a_i$, then $c'_i(t)=\lambda'_i(t)/\sqrt{-1}$ is
bounded between the minimum $m_1$ and maximum $m_2$ for all the
$b_i$'s.

When the matrix changes continuously, the eigenvalues also vary continuously.
To prove the lemma for general $P$ and $Q$, we only need to notice the fact that
generically the unitary matrices have no multiple eigenvalues. In
fact, the set of those unitary matrices with multiple eigenvalues form a
real subvariety of codimension $3$. Therefore for generically chosen $P$ and
$O$, then $1$-parameter curve $P\exp (tB)$ has no intersection
with it. If they intersect at finite points, it does not matter
either. Though the eigenvalue functions are not globally smoothly
defined, in each small closed interval, they can be continuous
defined, and in each small open interval, they are smooth. The
argument can still be carried out for each interval, and one can get the
estimate of the eigenvalues for $t=1$.

The following lemma is the essential technique to prove the claim in
section \ref{s2n+1-section} that the Clifford Wolf homogeneous
Randers metrics we constructed on $S^{2n+1}$ satisfying the property
that the geodesics starting from one point will all pass another
point, and they do not intersect each other in the midway.

\begin{lemma}
Let $U\in U(m+l)$ and $X=\sqrt{-1}\mbox{diag}(-I_l,I_m)$. If
the commutator $\exp (tX)U\exp (-tX)U^*$ has an eigenvalue
$1$ for some $t\in (0,\pi)$, then for each $t\in (0,\pi)$ it has an
eigenvalue $1$ with the same eigenvector.
\end{lemma}
\begin{proof}
We denote $\mbox{exp}(tX)$ as $\left(
                                        \begin{array}{cc}
                                          \lambda & 0 \\
                                          0 & \bar{\lambda} \\
                                        \end{array}
                                      \right)
$, $\lambda\neq 1$ or $-1$, and $U$ as $\left(
                                          \begin{array}{cc}
                                            U_1 & U_2 \\
                                            U_3 & U_4 \\
                                          \end{array}
                                        \right)
$. Then
\begin{equation}
\mbox{exp}(tX)U\mbox{exp}(-tX)U^*=I+\left(
                                      \begin{array}{cc}
                                        (\lambda^2-1)U_2 & 0 \\
                                        0 & (\bar{\lambda^2}-1)U_3 \\
                                      \end{array}
                                    \right)
\left(
  \begin{array}{cc}
    U^*_2 & U^*_4 \\
    U^*_1 & U^*_3 \\
  \end{array}
\right).\nonumber
\end{equation}
It is not hard to see that if $1$ is an eigenvalue, then $U_2$ or $U_3$
must be singular. The eigenspace of $1$ is the direct sum of the
kernel  of $U_2$ and $U_3$, multiplied by an invertible matrix
irrelative to $t$. Obviously the change of $\lambda\in S^1\backslash
\{\pm 1\}$ or $t\in (0,1)$ does not affect the eigenspace for $1$.
\end{proof}

\noindent {\bf Acknowledgement. }\quad We are grateful to Dr. Libing Huang and Dr. Zhiguang Hu for useful discussions. This work was finished during the second author's visit to the Chern institute of Mathematics. He would like to express his deep gratitude to the members of the institute for their hospitality.

\end{document}